\documentclass{article}
\usepackage{graphicx} 
\usepackage[utf8]{inputenc}
\usepackage[english]{babel}
\usepackage{amsmath}
\usepackage{amsfonts}
\usepackage{amssymb}
\usepackage{hyperref}
\usepackage{amsthm}
\usepackage[left=3cm,right=3cm,top=3cm,bottom=3cm]{geometry}

\theoremstyle{plain}
\newtheorem{teo}{}[section]
\newtheorem{prop}[teo]{Proposition}
\newtheorem{cor}[teo]{Corollary}

\newtheorem{thm}[teo]{Theorem}
\theoremstyle{definition}

\newcommand\blfootnote[1]{%
  \begingroup
  \renewcommand\thefootnote{}\footnote{#1}%
  \addtocounter{footnote}{-1}%
  \endgroup
}

\title{Semiflows deforming automorphisms groups}
\author{Pedro J. Chocano}
\date{}

\begin{document}

\maketitle

\begin{abstract}

In this short note, we prove that, for any pair of finite groups $G$ and $H$, there exist a finite $T_0$-space $X$ and a semiflow $\varphi\colon [0,\infty)\times X\to X$ such that $\operatorname{Aut}(X)\cong G$, whereas $\operatorname{Aut}(\varphi_t(X))\cong H$ for every $t>0$. Thus, the symmetry group of a finite $T_0$-space and those of all its positive-time images can be prescribed independently.

\end{abstract}
\blfootnote{2020  Mathematics  Subject  Classification: 37C10, 	22F50.}
\blfootnote{Keywords: automorphism group, semilfow, topological space.}

\blfootnote{This research is partially supported by Grant  PID2021-126124NB-100 from Ministerio de Ciencia, Innovación y Universidades (Spain) and  2025/SOLCON-159637 from Rey Juan Carlos University}
%\blfootnote{This research is partially supported by Grant  PID2021-126124NB-100 from Ministerio de Ciencia, Innovación y Universidades (Spain).}

\section{Introduction}

The realization of algebraic structures as automorphism groups of mathematical objects is a classical theme in mathematics. One of the earliest and best-known examples is the realization of a finite group as the automorphism group of a suitably constructed graph or colored directed graph, for instance by means of Cayley-type constructions. This problem has inspired a wide range of developments, including topological constructions such as those presented in~\cite{groot2010groups,groot158groups}. A broad overview of realization problems for automorphism groups in different mathematical contexts can be found in~\cite[Section~4]{Babai1995}. In this paper, we study a realization problem with a dynamical flavor. More precisely, we consider the following question. \begin{quote} \emph{Given two finite groups $G$ and $H$, does there exist a finite topological $T_0$-space $X$ and a semiflow \[ \varphi\colon \mathbb{R}_0^+\times X\longrightarrow X \] such that \[ \operatorname{Aut}(X)\cong G \qquad\text{and}\qquad \operatorname{Aut}(\varphi_t(X))\cong H \] for every $t>0$, where $\varphi_t(x)=\varphi(t,x)$?} \end{quote} We provide an affirmative answer to this question. In fact, our main result shows that the automorphism group of the initial space and the automorphism groups of all its positive-time images can be prescribed independently. More precisely, we prove the following. 

\begin{thm}\label{thm_main_result} For any pair of finite groups $G$ and $H$, there exist a connected finite topological $T_0$-space $X$ and a semiflow \[ \varphi\colon \mathbb{R}_0^+\times X\longrightarrow X \] such that \[ \operatorname{Aut}(X)\cong G \qquad\text{and}\qquad \operatorname{Aut}(\varphi_t(X))\cong H \] for every $t>0$, where $\varphi_t(x)=\varphi(t,x)$. \end{thm} 

The construction starts with realizations of $G$ and $H$ as automorphism groups of finite $T_0$-spaces and connects them through a semiflow that collapses a suitable part of the space at every positive time. As a consequence, the automorphism group changes from $G$ to $H$, while the image $\varphi_t(X)$ remains independent of $t>0$.

We do not address here the problem of minimizing either the cardinality or the height of the space $X$. Nevertheless, the construction used in the proof of Theorem~\ref{thm_main_result} yields explicit bounds. The resulting space has \[ |X| = 1+|G|(|S_G|+2)+|H|(|K_H|+2)+\sum_{i=1}^{|H|} i, \] where $S_G$ and $K_H$ are generating sets of $G$ and $H$, respectively. Equivalently, \[ |X| = 1+|G|(|S_G|+2)+|H|(|K_H|+2)+\frac{|H|(|H|+1)}{2}. \] Since $|S_G|\le |G|$ and $|K_H|\le |H|$, we obtain the coarse estimate \[ |X| \le |G|^2+2|G| +\frac{3|H|^2+5|H|}{2} +1. \] Thus, our construction provides an explicit quadratic upper bound in terms of the orders of $G$ and $H$. For finite groups $G$ and $H$, let $\sigma(G,H)$ denote the minimum cardinality of a connected finite $T_0$-space admitting a semiflow with \[ \operatorname{Aut}(X)\cong G \qquad\text{and}\qquad \operatorname{Aut}(\varphi_t(X))\cong H \] for every $t>0$. The above construction shows that $\sigma(G,H)\leq |X|$. Determining the exact value of $\sigma(G,H)$, or obtaining sharper upper and lower bounds, remains an interesting open problem. The space produced by our construction has height \[ \operatorname{ht}(X)=|S_G|+|H|+4, \] where the height of a finite $T_0$-space is the maximum number of strict inequalities in a chain.

Our construction makes use of the finite $T_0$-space introduced in \cite{barmak2009automorphism}. In that paper, the authors showed that every finite group $G$ can be realized as the automorphism group of a finite $T_0$-space with $|G|(|S_G|+2)$ points, where $S_G$ is a generating set of $G$. Although this construction is not optimal, several improvements have subsequently appeared. For example, the construction in \cite{barmak2020automorphism2} requires only $4|G|$ points, while the one in \cite{babai1980finite} uses $3|G|$ points. More recently, the minimum number of points required to realize a finite cyclic group was determined in \cite{barmak2024smallest}. Consequently, by replacing the initial realization of $G$ and $H$ with some of these more economical constructions and adapting our methods accordingly, one could obtain better bounds for $|X|$. We have opted for the approach of \cite{barmak2009automorphism} because it leads to a simpler and more transparent proof. A similar phenomenon occurs with respect to height. For instance, \cite{costoya2018realizability} provides realizations of arbitrary finite groups by posets of height~$1$.

The space $X$ arising from the proof of Theorem~\ref{thm_main_result} is contractible. However, the construction can be modified so as to prescribe the weak homotopy type of $X$. More precisely, we obtain the following consequence.

\begin{cor}\label{cor_prescribed_weak_type} Let $G$ and $H$ be finite groups, and let $K$ be a finite simplicial complex. Then there exist a connected finite $T_0$-space $X$ and a semiflow \[ \varphi\colon \mathbb{R}_0^+\times X\longrightarrow X \] such that \[ \operatorname{Aut}(X)\cong G, \qquad \operatorname{Aut}(\varphi_t(X))\cong H \] for every $t>0$, and $X$ is weakly homotopy equivalent to $|K|$. \end{cor}

The construction relies on two main ingredients. The first is the theory of finite topological spaces, which has proved to be a useful framework for addressing realization problems; see, in chronological order, \cite{barmak2009automorphism,chocano2020topological, chocano2020some,barmak2020automorphism2, viruel2024permutation,chocano2025realizing, barmak2024smallest}. The second is provided by recent developments in the study of dynamical systems on finite topological spaces; see, also in chronological order, \cite{barmak2011lefschetz,BarmakMrozekWanner2024conley, ChocanoMoronRuiz2025,Chocano2025semiflows} and the references therein. The paper is organized as follows. In Section~\ref{sec_preliminaries}, we review the necessary background on finite topological spaces and semiflows. In Section~\ref{sec_main_result}, we give the construction and prove the main results.

\section{Preliminaries}\label{sec_preliminaries}

We refer the reader to~\cite{barmak2011algebraic} for an excellent exposition of the theory of finite topological spaces. In this section, we recall and introduce the essential notions required to prove the main result. All of them can be found in, or deduced from, the aforementioned reference.

The following result is fundamental in the theory of finite topological spaces.

\begin{thm}
    The category of finite $T_0$-spaces (with continuous maps) is isomorphic to the category of finite partially ordered sets (with order-preserving maps). 
\end{thm}

This result allows us to identify objects from both categories and establish equivalences between various notions. From now on, we will not distinguish between a finite topological $T_0$-space and a finite partially ordered set (or poset), and we will treat them as the same object without further mention.

Let $f, g : X \rightarrow Y$ be two continuous maps between finite $T_0$-spaces. We write $f \leq g$ whenever $f(x) \leq g(x)$ for every $x \in X$.

\begin{thm}
Let $f, g : X \rightarrow Y$ be two continuous maps between finite $T_0$-spaces. Then $f$ is homotopic to $g$ if and only if there exists a finite sequence of continuous maps $f_i : X \rightarrow Y$ for $i = 0, \ldots, n$ such that $f_0 = f$, $f_n = g$, and $f_i \geq f_{i+1}$ or $f_i \leq f_{i+1}$ for all $i = 0, \ldots, n-1$.
\end{thm}

Given a finite $T_0$-space $X$ and a point $x \in X$, we denote by $U_x$ ($F_x$)the intersection of all open (closed) sets containing $x$. Equivalently, $U_x = \{ y \in X \mid y \leq x \}$ ($F_x= \{ y \in X \mid x \leq y \}$).  The \emph{height} of $X$, denoted by $\textnormal{ht}(X)$, is one less than the maximum number of elements in a chain of $X$. The height of a point $x \in X$ is defined as $\textnormal{ht}(x) := \textnormal{ht}(U_x)$.

The \emph{Hasse diagram} of $X$ is a directed graph whose vertices are the points of $X$, with a directed edge from $x$ to $y$ if and only if $x < y$ and there is no $z$ such that $x < z < y$.

Let $Y$ be a subspace of a topological space $X$, and let $i:Y\hookrightarrow X$ be the inclusion. We say that $Y$ is a \emph{strong deformation retract} of $X$ if there exist a map $r:X\rightarrow Y$ and a homotopy $H:X\times [0,1]\rightarrow X$ from $\mathrm{id}_X$ to $i\circ r$ such that $r\circ i=\mathrm{id}_Y$ and $H(y,t)=y$ for every $y\in Y$ and $t\in [0,1]$. For simplicity, we shall sometimes refer to $r$ as a \emph{strong deformation retraction}, although this departs from the standard terminology, in which a strong deformation retraction refers to the homotopy $H$ rather than to the map $r$.

A point $x \in X$ is called a \emph{down beat point} (\emph{up beat point}) if $U_x \setminus \{x\}$ ($F_x \setminus \{x\}$) has a maximum (minimum), which will be denoted by $\overline{x}$ ($\underline{x}$). \begin{thm}\label{thm_strong_retraction} Let $X$ be a finite $T_0$-space and let $y\in X$ be a down beat point (up beat point). Then the map $r:X\to X\setminus\{y\}\subseteq X$ defined by $r(x)=x$ for $x\neq y$ and $r(y)=\overline{y}$ ($r(y)=\underline{y}$) is a strong deformation retraction. \end{thm} 
Note that, under the hypotheses of Theorem~\ref{thm_strong_retraction} for the case of a down beat point, we have $\mathrm{id}_X \ge r$, where $\mathrm{id}_X$ denotes the identity map on $X$.

A point $x \in X$ is called a \emph{weak down beat point} (\emph{weak up beat point}) if $U_x \setminus \{x\}$ ($F_x \setminus \{x\}$) is contractible. 

\begin{thm}
    Let $X$ be a finite $T_0$-space and $x$ a weak beat point. Then $X\setminus \{x\}$ has the same weak homotopy type as $X$.
\end{thm}

Note that every beat point is a weak beat point, but the converse is false.

Let us recall in the following proposition some trivial consequences of the notion of homeomorphisms within this setting.

\begin{prop}\label{prop_properties_homeomorphism} Let $X$ be a finite $T_0$-space and let $f:X\to X$ be a homeomorphism. \begin{enumerate} \item If $x\in X$ is a down beat point (up beat point), then $f(x)$ is also a down beat point (up beat point). \item If $x\in X$ is a weak beat point that is not a beat point, then $f(x)$ is a weak beat point that is not a beat point. \item For every $x\in X$, $\mathrm{ht}(x)=\mathrm{ht}(f(x))$. \end{enumerate} \end{prop}

Let $X$ and $Y$ be two finite $T_0$-spaces. The \emph{non-Hausdorff join} $X\circledast Y$ is the disjoint union $X\sqcup Y$ keeping the given ordering within $X$ and $Y$ and setting $x\leq y$ for every $x\in X$ and $y\in Y$. 
\begin{prop}\label{prop_automorfismos_join}
    Let $X$ and $Y$ be two finite $T_0$-spaces. Then $\textnormal{Aut}(X\circledast Y)=\textnormal{Aut}(X)\times \textnormal{Aut}(Y)$.
\end{prop}

We now recall the main results from~\cite{Chocano2025semiflows}. A \emph{semiflow} $\varphi$ on a topological space $X$ is a continuous map
\[
\varphi : \mathbb{R}_0^+ \times X \rightarrow X,
\]
where $\mathbb{R}_0^+$ denotes the set of non-negative real numbers, satisfying:
\begin{enumerate}
    \item $\varphi(0, x) = x$ for every $x \in X$,
    \item $\varphi(t, \varphi(s, x)) = \varphi(t + s, x)$ for every $x \in X$ and $s, t \in \mathbb{R}_0^+$.
\end{enumerate}
For simplicity, given $t \in \mathbb{R}_0^+$, we denote by $\varphi_t$ the map $\varphi(t, \cdot) : X \rightarrow X$.

\begin{thm}\label{thm_semiflow_existence}
    Let \( X \) be a finite $T_0$-space. If \( X \) does not have down beat points and \(\varphi : \mathbb{R}_{0}^+ \times X \to X \) is a semiflow, then \( \varphi_0=\varphi_t \) for every $t\in \mathbb{R}_0^+$.
\end{thm}
Thus, the presence of down beat points is a crucial condition for defining non-trivial semiflows on finite spaces.

\begin{thm}\label{thm_semiflow_description}
 Let \( X \) be a finite  $T_0$-space and let \(\varphi : \mathbb{R}_{0}^+ \times X \to X \) be a non-trivial semiflow. Then $\varphi_t=r$ for every positive real value $t$ where $r:X\rightarrow X$ is a strong deformation retraction such that $r\leq \varphi_0$.
\end{thm}

Note that the map $r$ obtained in Theorem~\ref{thm_strong_retraction} can be used to define a non-trivial semiflow. Under the hypotheses of that theorem, define \(\varphi : \mathbb{R}_{0}^+ \times X \to X \) as follows:
\[
\varphi(t, x) :=
\begin{cases}
x & \text{if } t = 0, \\
r(x) & \text{if } t > 0.
\end{cases}
\]

Indeed, we have the following result:

\begin{thm}\label{thm_semiflow_description_2}
    Let $X$ be a finite $T_0$-space and let $r:X\rightarrow X$ be a strong deformation retraction such that $r\leq \textnormal{id}_X$. Then the map \(\varphi : \mathbb{R}_{0}^+ \times X \to X \) define by $\varphi(0,x)=x$ and $\varphi(t,x)=r(x)$ whenever $t>0$ is a semiflow.
\end{thm}
\begin{proof}
   Let us prove the continuity of $\varphi$. Since $\varphi$ is constant on $(0,\infty)\times X$, it suffices to verify its continuity at points of the form $(0,x)$. Let $V$ be an open neighborhood of $\varphi(0,x)=x$. Then $U_x\subseteq V$. On the other hand, for any $\varepsilon>0$, the set $[0,\varepsilon)\times U_x$ is an open neighborhood of $(0,x)$ in $\mathbb{R}_0^+\times X$. Clearly, \[ \varphi([0,\varepsilon)\times U_x)=U_x\cup r(U_x). \] Since $r$ is continuous, $r(U_x)=U_{r(x)}$. Moreover, as $r\leq \mathrm{id}_X$, we have $r(U_x)\subseteq U_x$. Therefore, \[ \varphi([0,\varepsilon)\times U_x)\subseteq U_x\subseteq V, \] which proves the continuity of $\varphi$ at $(0,x)$. To conclude, it is clear that $\varphi$ satisfies the first condition of a semiflow. Furthermore, \[ \varphi(t,\varphi(s,x))=\varphi(t+s,x) \] holds trivially whenever $t=0$ or $s=0$. If $t,s>0$, then \[ \varphi(t,\varphi(s,x)) =\varphi(t,r(x)) =r(r(x)) =r(x) =\varphi(t+s,x), \] because $r$ is a retraction and therefore idempotent, that is, $r^2=r$.
\end{proof}

Thus, we have a complete characterization of semiflows within this context.

\section{Proof of the main results}\label{sec_main_result}

\begin{proof}[Proof of Theorem \ref{thm_main_result}
]
We briefly outline the key ideas behind its proof. The construction is inspired by the work of J.A. Barmak and E.G. Minian in~\cite{barmak2009automorphism}, which realizes any finite group as the automorphism group of a finite $T_0$-space. We apply this construction twice—once for each group—and then modify the resulting spaces appropriately to define a semiflow, using the characterization of non-trivial semiflows given in Theorem~\ref{thm_semiflow_description} and Theorem~\ref{thm_semiflow_description_2}.
\begin{figure}
    \centering
    \includegraphics[width=0.5\linewidth]{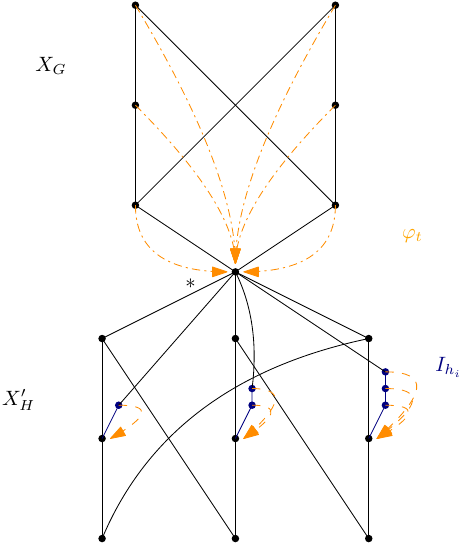}
    \caption{Hasse diagram of $X$ and a semiflow $\varphi$ such that $\textnormal{Aut(X)}\cong \mathbb{Z}_2$ and $\textnormal{Aut}(\varphi_t(X))\cong \mathbb{Z}_3$ for any $t>0$. }
    \label{fig:hasse_diagram}
\end{figure}

Let \( S_G = \{s_1, \ldots, s_n\} \) and \( K_H = \{k_1, \ldots, k_m\} \) be generating sets for the groups \( G \) and \( H \), respectively. Define $X_G=G\times \{-1,0,1,...,n \}$ and consider the following partial order on \( X_G \):
\begin{enumerate}
    \item \( (g, i) \leq (g, j) \) if \( -1 \leq i \leq j \leq n \),
    \item \( (g, -1) \leq (g', i) \) if \( g' = g s_i^{-1} \) for \( 1 \leq i \leq n \).
\end{enumerate}
It is straightforward to verify that this defines a partial order on \( X_G \). Repeat the same construction for \( H \), defining $X_H=H\times \{-1,0,1,...,m \}$ with the partial order:
\begin{enumerate}
    \item \( (h,i)\leq (h,j) \) if $-1\leq i\leq j\leq m$.
    \item \( (h,-1)\leq (h',i)  \) if $h'=hk_i^{-1}$ where $1\leq i\leq m$.
\end{enumerate}

As shown in~\cite{barmak2009automorphism}, these constructions yield finite $T_0$-spaces with automorphism groups isomorphic to the respective groups: \( \textnormal{Aut}(X_G) \cong G\) and \( \textnormal{Aut}(X_H) \cong H\).

Next, label the elements of \( H \) as \( H = \{h_i\}_{i=1}^{|H|} \). For each \( h_i \in H \), define a totally ordered set
\[
I_{h_i} = \{j^i_1, j^i_2, \ldots, j^i_i\} \quad \text{with} \quad j^i_1 < j^i_2 < \cdots < j^i_i.
\]
Now consider the disjoint union
\[
X_H' = X_H \bigsqcup \left( \bigsqcup_{i=1}^{|H|} I_{h_i} \right),
\]
where we preserve the internal orderings and extend the partial order by declaring
\[
(h_i, 0) < j^i_1
\]
for each $i=1,\ldots, |H|$. The remaining relations are obtained by transitivity. Note that each \( j^i_i \in I_{h_i} \) is the unique maximal down beat point of height \( i \) in $X_H'$.
By Proposition~\ref{prop_properties_homeomorphism}, this implies that \( \textnormal{Aut}(X_H')\) is the trivial group.

Now define
\begin{align*}
    X=X_H'\circledast \{* \}\circledast X_G,
\end{align*}
where \( \{*\} \) denotes the one-point space. By Proposition~\ref{prop_automorfismos_join}, we have \(  \textnormal{Aut}(X)\cong G  \). Define the map \( \varphi : \mathbb{R}_0^+ \times X \rightarrow X \) by
\[
\varphi(t, x) :=
\begin{cases}
* & \text{if } x \in X_G \cup \{*\}, \\
x & \text{if } x \in X_H, \\
(h_i, 0) & \text{if } x \in I_{h_i},
\end{cases}
\]
for \( t > 0 \), and set \( \varphi(0, x) := x \). Figure~\ref{fig:hasse_diagram} illustrates the Hasse diagram of \( X \) for the case \( G = \mathbb{Z}_2 \) and \( H = \mathbb{Z}_3 \), where the posets \( I_{h_i} \) are shown in blue and the semiflow \( \varphi \) is schematically represented in orange.

By construction,  $\textnormal{Aut}(\varphi(0, X)) \cong G$ and $\textnormal{Aut}(\varphi(t, X)) \cong H$ for all $t > 0$, since \( \varphi(0, X) = X \) and \( \varphi(t, X) = X_H\circledast \{* \} \) for any \( t > 0 \). It remains to verify that \( \varphi \) is indeed a semiflow. Let \( r := \varphi_t \) for \( t > 0 \). We show that \( r \) is a strong deformation retraction and that \( \textnormal{id}_X \geq r \). Since \( * < x \) for any \( x \in X_G \), we have \( x \geq r(x) = * \). Moreover, for any \( j^i_k \in I_{h_i} \) with \( 1 \leq k \leq i \), we have \( j^i_k \geq r(j^i_k) = (h_i, 0) \). Thus, \( \textnormal{id}_X \geq r \). Clearly, $r$ preserves the order of the poset $X$, which implies that $r$ is a continuous map. By the definition of $r$, \( r \circ i = \textnormal{id}_{X_H\circledast\{*\}} \), so \( r \) is a strong deformation retraction under the hypotheses of Theorem~\ref{thm_semiflow_description_2}. Therefore, \( \varphi \) is a semiflow.
\end{proof}

\begin{proof}[Proof of Theorem \ref{cor_prescribed_weak_type}]

Our starting point is the topological space $X$ constructed above, together with one of the constructions introduced in \cite{chocano2020some}. Recall that, in \cite{chocano2020some}, the authors provide a method for constructing finite $T_0$-spaces having the prescribed weak homotopy type of a given simplicial complex and a trivial homeomorphism group. One of the key ideas underlying this method is the construction of an infinite family of pairwise non-homeomorphic finite $T_0$-spaces with trivial homeomorphism groups, all of which have the weak homotopy type of a point. Each member of this family does not have beat points and contains weak beat points that are not beat points. This distinction will be crucial in the argument below. Thus, starting with a finite $T_0$-space having the desired weak homotopy type, for instance, the face poset of the given simplicial complex, one attaches a different member of this family to each point. The resulting space has the same weak homotopy type as the original space and a trivial homeomorphism group. Let $Y$ be a finite $T_0$-space with the same weak homotopy type as $|K|$ and a trivial homeomorphism group, constructed using the method introduced in \cite{chocano2020some}. If $|K|$ is contractible, then we may use the same arguments from the proof of Theorem \ref{thm_main_result} because $X$ is contractible. If $K$ is not contractible, then $Y$ contains at least one weak beat point. Consider \[ Z=X\sqcup Y\sqcup \{\overline{*} \},\] where we retain the given partial orders on $X$ and $Y$ and extend them by declaring that \[ y< \overline{*}>* \] for a chosen $y\in Y$ such that its height equals to $0$, with all remaining relations determined by transitivity. It follows from Proposition \ref{prop_properties_homeomorphism} that \[ \operatorname{Aut}(Z)\cong G, \] since $X_H'$ has no weak beat points and $Y$ does not have beat points. We extend the map $r$ defined in the proof of Theorem \ref{thm_main_result} to $Z$ by declaring that its restriction to $Y\sqcup\{\overline{*} \}$ is the identity map. By construction, the space $Z$ and the extended map $r$ satisfy all the required properties. To complete the proof, note that $Z$ has the same homotopy type as $X_H'\circledast\{* \}\sqcup Y \sqcup\{\overline{*}\}$ since $*$ is a minimum of $\{ *\}\circledast X_G$. For the same reason, $X_H'\circledast\{* \}\sqcup Y \sqcup\{\overline{*}\}$ is homotopy equivalent to $\{* \}\sqcup Y \sqcup \{\overline{*}\}$. Now, $*$ is a beat point and after removing it $\overline{*}$ is a beat point. Removing the previous point we obtain that $Z$ is homotopy equivalent to $Y$, which gives that $Z$ has the same weak homotopy type as $|K|$. \end{proof}
 
\bibliography{bibliografia}
\bibliographystyle{plain}

\newcommand{\Addresses}{{%additional braces for segregating \footnotesize
  \bigskip
  \footnotesize

  \textsc{ P.J. Chocano, Departamento de Matemática Aplicada, Ciencia e Ingeniería de los Materiales y Tecnología Electrónica, ESCET Universidad Rey Juan Carlos, 28933 Móstoles (Madrid), Spain}\par\nopagebreak
 \textit{E-mail address}:\texttt{pedro.chocano@urjc.es}

}}

\Addresses
\end{document}